\documentclass[11pt, reqno, twoside, makeidx]{amsart}
\usepackage[margin=2.7cm, marginpar=1cm]{geometry}
\usepackage{wrapfig}
\usepackage{marginnote}
\usepackage[english]{babel}
\usepackage[T1]{fontenc}
\usepackage{amsmath,amsthm,amssymb,amsfonts}
\numberwithin{equation}{section}
\usepackage{enumerate}
\usepackage{enumitem}
\usepackage{faktor}
\usepackage{dsfont}
\usepackage{scalerel,stackengine}
\usepackage{textcomp}
\usepackage{graphicx}
\usepackage{extarrows}
\usepackage{nicematrix}
\usepackage{epigraph}
\usepackage{sidecap}
\usepackage{indentfirst}
\usepackage{tikz}
\usepackage{tikz-cd}
\usepackage{tkz-euclide}
\usetikzlibrary{shapes.misc,arrows,matrix,patterns,decorations.markings,positioning}
\tikzset{cross/.style={cross out, draw=black, minimum size=2*(#1-\pgflinewidth), inner sep=0pt, outer sep=0pt},
cross/.default={4.5pt}}
\usepackage{pgfplots}
\usepackage{caption}
\usepackage{subcaption}
\usepackage{float}
\usepackage{color}
\usepackage{xcolor}
\definecolor{burgundy}{rgb}{0.5,0.0,0.13}
\usepackage{epstopdf}
\usepackage{epsfig}
\usepackage{stmaryrd}
\usepackage{multicol}
\usepackage{verbatim}
\usepackage{hyperref}

\definecolor{b}{RGB}{27,77,137}
\definecolor{o}{RGB}{211,84,0}
\definecolor{g}{RGB}{13,113,3}
\definecolor{r}{RGB}{128,9,27}

\DeclareMathOperator{\tb}{tb}
\DeclareMathOperator{\rot}{rot}

\renewcommand{\geq}{\geqslant}
\renewcommand{\leq}{\leqslant} 
\renewcommand{\epsilon}{\varepsilon}

\newcommand{\Z}{\mathbb{Z}}
\newcommand{\Q}{\mathbb{Q}}

\newcommand{\dd}{\mathrm{d}}

\newcommand{\x}{\mathbf{x}}

\newtheorem{teo}{Theorem}[section]
\newtheorem*{teo*}{Theorem}
\newtheorem{lemma}[teo]{Lemma}
\newtheorem{prop}[teo]{Proposition}
\newtheorem*{prop*}{Proposition}

\newtheorem{cor}[teo]{Corollary}

\newtheorem{remark}[teo]{Remark}

\usepackage{xpatch}
\makeatletter
\xpatchcmd{\@thm}{\thm@headpunct{.}}{\thm@headpunct{}}{}{}
\makeatother

\pgfplotsset{compat=1.18}
\begin{document}
\title[Stein fillable contact structures with half Giroux torsion]{Stein fillable contact structures with \\ half Giroux torsion}
\author{Alberto Cavallo and Irena Matkovi\v c}
\address{Uppsala Universitet, Uppsala 751 06, Sweden}
\email{irma6504@student.uu.se}
\email{acavallo@impan.pl}

\begin{abstract}
 We show that there exist Stein fillable contact 3-manifolds which contain a separating half Giroux torsion layer. Our examples are on Seifert fibred spaces over the sphere. 
 \vspace{-0.3cm}
\end{abstract}

\maketitle

\thispagestyle{empty}

\section{Introduction}
A conjecture of Ghiggini asked whether a contact 3-manifold which contains a layer of separating half Giroux torsion necessarily had vanishing contact invariant. This has been disproved  by Min and Varvarezos in \cite[Theorem 1.2]{MV}, where they showed that the splice of the complement of the negative trefoil with the one of any torus knot $T_{p,-q}$ admits such a structure.  
It remained open to investigate whether the presence of half Giroux torsion obstructs fillability; in particular, they explicitly ask if this could hold at least for Seifert fibred spaces, see \cite[Questions 1.7 and 1.11]{MV}. We answer these questions in the negative.

It is well known \cite{Eliashberg} that a half Giroux torsion layer along a non-separating torus does not obstruct Stein fillability; note that a non-separating torus can only appear in a 3-manifold with positive first Betti number. In this paper a half torsion layer is always separating unless differently specified.

\begin{teo}
 \label{teo:main}
 There exist Stein fillable contact structures on Seifert fibred spaces which contain separating half Giroux torsion. 
\end{teo}

For a Seifert fibred space $M:=M(e_0;r_1,...,r_n)$ we fix the convention that $e_0\in\Z$ and $r_i\in(0,1)$ is rational for $i=1,...,n$. We recall that a tight contact structure on $M$ can be either negative- or zero-twisting, see \cite{Honda2}. 

Lisca and Stipsicz \cite{LS.III} showed that when $n=3$ and $e_0=-1$ any zero-twisting contact structure on $M$ is presented by contact surgery on a $T_{5,-5}$ torus link with two coefficients equal to $(+1)$, and the other ones equal to $(-\frac{1}{r_i})$ for $i=1,2,3$. 
Our structures belong to a family obtained by extending their construction.

\begin{figure}[ht]
 \begin{tikzpicture}[scale=0.7]
    \tkzDefPoints{0/0/A, 1.5/1/B, 1.5/-1/D, 3/-1/E, 5/-1/F, 3/1/H, 5/1/I}  \tkzDefPoints{2.25/-0.1/X'', 2.25/0/Y'', 2.25/0.1/Z''}
    \tkzDefPoints{3.5/-1/X, 4.5/-1/Y}\tkzDefPoints{3.9/-1/P, 4/-1/Q, 4.1/-1/R}\tkzDefPoints{3.5/1/X', 4.5/1/Y'}\tkzDefPoints{3.9/1/P', 4/1/Q', 4.1/1/R'}
    \tkzDrawPoints[fill,black,size=1](P,Q,R)\tkzDrawPoints[fill,black,size=1](P',Q',R')\tkzDrawPoints[fill,black,size=1](X'',Y'',Z'')
    \tkzDrawSegment(A,B)\tkzDrawSegment(A,D)\tkzDrawSegment(E,D)\tkzDrawSegment(E,X)\tkzDrawSegment(Y,F)\tkzDrawSegment(B,H)
    \tkzDrawSegment(H,X')\tkzDrawSegment(Y',I)
    \tkzDrawPoints[fill,black,size=5](A,B,D,E,F,H,I)
     \tkzLabelPoint[above left](A){$e_0$} \tkzLabelPoint[above right](B){$m_1^1$}\tkzLabelPoint[above right](H){$m_2^1$}\tkzLabelPoint[above right](I){$m_{k_1}^1$}
     \tkzLabelPoint[below right](D){$m_1^n$}\tkzLabelPoint[below right](E){$m_2^n$}\tkzLabelPoint[below right](F){$m_{k_n}^n$}
\end{tikzpicture}\hspace{1cm}
\begin{tikzpicture}
\begin{scope}[scale=1.5]
\begin{scope}[shift={(0,.2)}]
\draw [domain=-0.5*pi:1.5*pi, scale=0.5] plot (\x, {1.5*sin(\x r)});
\draw [domain=-0.5*pi:1.5*pi, scale=0.5, postaction={decorate}] plot (\x, {-0.4*sin(\x r)-1.9});
\draw [fill=white,white] (1.75,-0.8) rectangle (1.85,-0.9);
\draw [fill=white,white] (1.55,-0.9) rectangle (1.65,-0.98);
\draw [fill=white,white] (1.4,-0.95) rectangle (1.47,-1.1);
\draw [fill=white,white] (1.15,-1) rectangle (1.25,-1.1);
\draw [scale=0.5] (-0.5*pi,-1.5) circle (0.1pt) node[left,scale=0.9]{$\left(-\frac{1}{r_n}\right)$};
\end{scope}
\draw[white] (.75,-1.1) circle (0pt) node[right,black,scale=0.9]{$\vdots$};
\begin{scope}[shift={(.06,-.3)},scale=0.95]
\draw [domain=-0.5*pi:-0.2*pi, scale=0.5] plot (\x, {1.5*sin(\x r)});
\draw [domain=-0.15*pi:1.5*pi, scale=0.5] plot (\x, {1.5*sin(\x r)});
\draw [domain=-0.5*pi:1.5*pi, scale=0.5, postaction={decorate}] plot (\x, {-0.4*sin(\x r)-1.9}) node[right,scale=0.9]{$\left(-\frac{1}{r_2}\right)$};
\draw [fill=white,white] (1.9,-0.8) rectangle (2,-0.9);
\draw [fill=white,white] (1.65,-0.85) rectangle (1.75,-0.98);
\draw [fill=white,white] (1.45,-0.95) rectangle (1.55,-1.1);
\end{scope}
\begin{scope}[shift={(.12,-.6)},scale=0.9]
\draw [domain=-0.5*pi:-0.3*pi, scale=0.5] plot (\x, {1.5*sin(\x r)});
\draw [domain=-0.22*pi:-0.1*pi, scale=0.5] plot (\x, {1.5*sin(\x r)});
\draw [domain=-0.05*pi:1.5*pi, scale=0.5] plot (\x, {1.5*sin(\x r)});
\draw [domain=-0.5*pi:1.5*pi, scale=0.5, postaction={decorate}] plot (\x, {-0.4*sin(\x r)-1.9});
\draw [fill=white,white] (1.85,-0.8) rectangle (1.95,-0.9);
\draw [fill=white,white] (1.6,-0.85) rectangle (1.7,-0.98);
\draw [scale=0.5] (-0.5*pi,-1.5) circle (0.1pt) node[left,scale=0.9]{$\left(-\frac{1}{r_1}\right)$};
\end{scope}
\begin{scope}[shift={(.18,-.9)},scale=0.85]
\draw [domain=-0.5*pi:-0.3*pi, scale=0.5] plot (\x, {1.5*sin(\x r)});
\draw [domain=-0.22*pi:-0.16*pi, scale=0.5] plot (\x, {1.5*sin(\x r)});
\draw [domain=-0.1*pi:0*pi, scale=0.5] plot (\x, {1.5*sin(\x r)});
\draw [domain=0.07*pi:1.5*pi, scale=0.5] plot (\x, {1.5*sin(\x r)});
\draw [domain=-0.5*pi:1.5*pi, scale=0.5, postaction={decorate}] plot (\x, {-0.4*sin(\x r)-1.9});
\draw [fill=white,white] (1.8,-0.8) rectangle (1.9,-0.9);
\end{scope}
\draw[white] (2.1,-1.8) circle (0pt) node[right,black,scale=0.9]{$\left.\begin{aligned}(+1)\\ (+1)\end{aligned}\right\}\ 1-e_0$};
\draw[white] (.75,-2) circle (0pt) node[right,black,scale=0.9]{$\vdots$};
\begin{scope}[shift={(.18,-1.3)},scale=0.85]
\draw [domain=-0.5*pi:-0.2*pi, scale=0.5] plot (\x, {1.5*sin(\x r)});
\draw [domain=-0.135*pi:-0.09*pi, scale=0.5] plot (\x, {1.5*sin(\x r)});
\draw [domain=-0.05*pi:0.01*pi, scale=0.5] plot (\x, {1.5*sin(\x r)});
\draw [domain=0.04*pi:0.19*pi, scale=0.5] plot (\x, {1.5*sin(\x r)});
\draw [domain=0.23*pi:1.5*pi, scale=0.5] plot (\x, {1.5*sin(\x r)});
\draw [domain=-0.5*pi:1.5*pi, scale=0.5, postaction={decorate}] plot (\x, {-0.4*sin(\x r)-1.9});
\end{scope}
\end{scope}
\end{tikzpicture}
\caption{\smaller[1]{Generalised Lisca-Stipsicz contact structures on $M(e_0;r_1,...,r_n)$ when $e_0\leq1$. We have that the smooth framings are determined by the continued fraction expansion $-\frac{1}{r_i}=m_1^i-\dfrac{1}{\ddots-\frac{1}{m_{k_i}^i}}=[m_1^i,...,m_{k_i}^i]$ where $m_j^i\leq-2$ for $j=1,...,k_i$ and $i=1,...,n$.}}
\label{fig:SFS}
\end{figure}
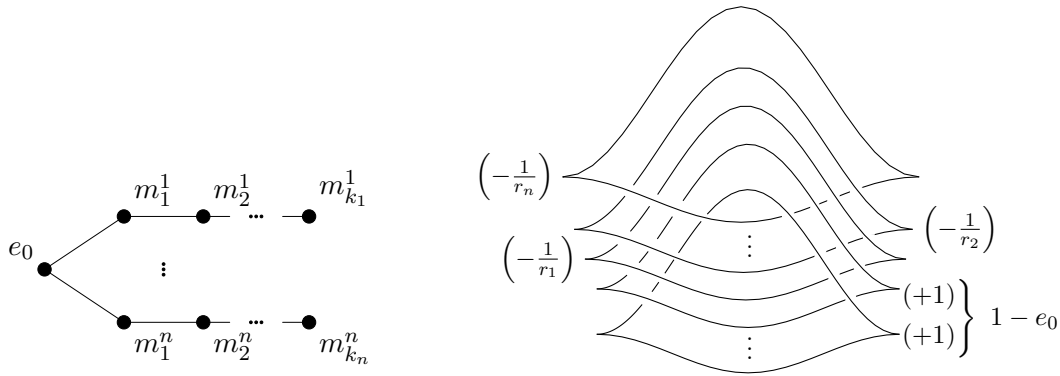

Suppose that $e_0\leq1$, we say that a contact structure on $M$ is a \emph{generalised Lisca-Stipsicz structure} when it can be written with a surgery presentation as in Figure \ref{fig:SFS}; namely, the number of $(+1)$-contact surgeries is $1-e_0$, while for each singular fibre we have contact $(-\frac{1}{r_i})$-surgery. Recall that such a diagram gives a family of contact structures, whose elements can be specified by replacing the rational surgery with a Legendrian surgery along a chain containing $k_i$ unknots. The $\tb$-numbers are determined by the continued fraction expansion; for the leading component $L^i_1$ it equals $m_1^i$, while on the remaining unknots $L^i_j$ in the chain, is given by $(m^i_2+1,...,m^i_{k_i}+1)$ for $i=1,...,n$ and $j>1$.
The rotation numbers are chosen arbitrarily in the set $\{\tb(L^i_j)+1,\tb(L^i_j)+3,\dots,-\tb(L^i_j)-1\}$.

Note that this family could be easily extended to the case when $e_0>1$, as done in \cite{GLS} for three legs, but we do not write it here because it is not needed for our purpose. Additionally, these structures are all supported by planar open books; in the paper we sketch how to translate a surgery presentation as in Figure \ref{fig:SFS} into an open book of this kind.

When $n>3$ not every zero-twisting tight structure is generalised Lisca-Stipsicz, as seen from the examples constructed by Ghiggini in \cite{G4}, but the converse still holds.

\begin{prop}
 \label{prop:zero}
 Every generalised Lisca-Stipsicz tight contact structure on $M(e_0;r_1,...,r_n)$ with $e_0\leq1$ is zero-twisting.   
\end{prop}

Note that for a contact structure being as in Figure \ref{fig:SFS} does not guarantee tightness, as it is well known from \cite{GLS-1,M.t} already in the $n=3$ case. 

We recall that a slightly different version of half Giroux torsion exists in literature; see \cite[Definition 2.3]{MV}. A \emph{half convex Giroux torsion} layer is an embedded $\pi$-twisting tight thickened torus with convex boundaries of the same slope and division number one. 
Such a concept is better aligned with convex decompositions and therefore usually easier to understand in Seifert fibred spaces; in addition, its relations with Stein fillability have been already studied by other authors\footnote{In \cite{SS} Shah and Simone already affirm the existence of a Stein fillable structure on a Seifert fibred space $Y$ of the form $M(-4;r_1,...,r_4)$, which should contain half (convex) Giroux torsion. They refer to a previous paper of Shah \cite{Shah}, but this seems to contain no justification for $Y$ to have such a torsion layer. 
Furthermore, we need to highlight that it now follows from our recent result in \cite[Theorem 1.1]{CM-negative} that every fillable structure on a negative-definite Seifert fibred space, such as $Y$, is necessarily negative-twisting; on the other hand, convex surface theory implies that the presence of a half convex Giroux torsion layer forces a tight contact structure to be zero-twisting.}.
We recall the following result from \cite[Lemma 2.3]{ACM}.

\begin{teo}[Alfieri-Cavallo-Matkovi\v c]
 \label{teo:ours}
 If $M(e_0;r_1,...,r_n)$ carries a zero-twisting tight contact structure without half convex Giroux torsion then $e_0\geq-1$.   
\end{teo}

Proposition \ref{prop:zero} and Theorem \ref{teo:ours} imply the following corollary.

\begin{cor}
 \label{cor:zero}
 Every generalised Lisca-Stipsicz tight structure on $M(e_0;r_1,...,r_n)$ with $e_0\leq-2$ necessarily contains a half convex Giroux torsion layer. 
\end{cor}

\begin{figure}[ht] 
\begin{tikzpicture}[scale=1] 
      \tkzDefPoints{0/0/A, 1.5/1.5/B, 1.5/-1.5/D, 1.5/0.5/C, 1.5/-0.5/G, 0/-3/O}      
       \tkzDrawSegment(A,B)\tkzDrawSegment(A,D)\tkzDrawSegment(A,C)\tkzDrawSegment(A,G)
    \tkzDrawPoints[fill,black,size=5](A,B,C,D,G)\tkzDrawPoints[fill,white,size=5](O)
     \tkzLabelPoint[above left](A){$-2$} \tkzLabelPoint[above right](B){$-2$}\tkzLabelPoint[above right](C){$-2$}
     \tkzLabelPoint[above right](D){$-2$}\tkzLabelPoint[above right](G){$-2$}
\end{tikzpicture}\hspace{3cm}
\begin{tikzpicture}
\begin{scope}[scale=1.9]
\begin{scope}[shift={(-.06,.6)}, scale=1.05]
\draw [domain=-0.5*pi:0.5*pi, scale=0.5] plot (\x, {1.5*sin(\x r)});
\draw [domain=0.5*pi:1.48*pi, scale=0.5] plot (\x, {1.3*sin(\x r)+0.2});
\draw [domain=-0.5*pi:1.5*pi, scale=0.5] plot (\x, {-0.6*sin(\x r)-2.1});
\begin{scope}[shift={(2.25,-.59)},scale=.12]
\draw [domain=-0.5*pi:0.5*pi, scale=0.5] plot (\x, {0.8*sin(\x r)-.1});
\draw [domain=-0.5*pi:0.5*pi, scale=0.5] plot (\x, {-0.8*sin(\x r)-1.8});
\end{scope}
\draw [shift={(.27,.07)}, fill=white,white] (1.75,-0.85) rectangle (1.8,-0.9);
\draw [shift={(-.1,-.12)}, fill=white,white] (1.75,-0.85) rectangle (1.8,-0.9);
\draw [shift={(-.3,-.22)}, fill=white,white] (1.75,-0.85) rectangle (1.8,-0.9);
\draw [shift={(-.45,-.3)}, fill=white,white] (1.75,-0.85) rectangle (1.8,-0.9);
\draw [shift={(-.55,-.37)}, fill=white,white] (1.75,-0.85) rectangle (1.8,-0.9);
\draw [shift={(-.68,-.43)}, fill=white,white] (1.75,-0.85) rectangle (1.8,-0.9);
\end{scope}

\begin{scope}[shift={(0,.2)}]
\draw [domain=-0.5*pi:-0.27*pi, scale=0.5] plot (\x, {1.5*sin(\x r)});
\draw [domain=-0.22*pi:0.5*pi, scale=0.5] plot (\x, {1.5*sin(\x r)});
\draw [domain=0.5*pi:1.48*pi, scale=0.5] plot (\x, {1.3*sin(\x r)+0.2});
\draw [domain=-0.5*pi:1.5*pi, scale=0.5] plot (\x, {-0.5*sin(\x r)-2});
\begin{scope}[shift={(2.25,-.59)},scale=.12]
\draw [domain=-0.5*pi:0.5*pi, scale=0.5] plot (\x, {0.8*sin(\x r)-.1});
\draw [domain=-0.5*pi:0.5*pi, scale=0.5] plot (\x, {-0.8*sin(\x r)-1.8});
\end{scope}
\draw [shift={(.11,.01)}, fill=white,white] (1.75,-0.85) rectangle (1.8,-0.9);
\draw [shift={(-.12,-.1)}, fill=white,white] (1.75,-0.85) rectangle (1.8,-0.9);
\draw [shift={(-.3,-.18)}, fill=white,white] (1.75,-0.85) rectangle (1.8,-0.9);
\draw [shift={(-.39,-.22)}, fill=white,white] (1.75,-0.85) rectangle (1.8,-0.9);
\draw [shift={(-.52,-.26)}, fill=white,white] (1.75,-0.85) rectangle (1.8,-0.9);
\end{scope}

\begin{scope}[shift={(.06,-.2)},scale=0.95]
\draw [domain=-0.48*pi:-0.34*pi, scale=0.5] plot (\x, {1.3*sin(\x r)+.2});
\draw [domain=-0.28*pi:-0.2*pi, scale=0.5] plot (\x, {1.3*sin(\x r)+.2});
\draw [domain=-0.15*pi:.5*pi, scale=0.5] plot (\x, {1.3*sin(\x r)+.2});
\draw [domain=0.5*pi:1.5*pi, scale=0.5] plot (\x, {1.5*sin(\x r)});
\draw [domain=-0.5*pi:1.5*pi, scale=0.5] plot (\x, {-0.4*sin(\x r)-1.9});
\begin{scope}[shift={(-.67,-.59)},scale=.12]
\draw [domain=-0.5*pi:0.5*pi, scale=0.5] plot (\x, {-0.8*sin(\x r)-.1});
\draw [domain=-0.5*pi:0.5*pi, scale=0.5] plot (\x, {0.8*sin(\x r)-1.8});
\end{scope}
\draw [shift={(.09,.02)}, fill=white,white] (1.75,-0.85) rectangle (1.8,-0.9);
\draw [shift={(-.12,-.06)}, fill=white,white] (1.75,-0.85) rectangle (1.8,-0.9);
\draw [shift={(-.23,-.09)}, fill=white,white] (1.75,-0.85) rectangle (1.8,-0.9);
\draw [shift={(-.34,-.12)}, fill=white,white] (1.75,-0.85) rectangle (1.8,-0.9);
\end{scope}

\begin{scope}[shift={(.12,-.6)},scale=0.9]
\draw [domain=-0.48*pi:-0.32*pi, scale=0.5] plot (\x, {1.3*sin(\x r)+.2});
\draw [domain=-0.25*pi:-0.16*pi, scale=0.5] plot (\x, {1.3*sin(\x r)+.2});
\draw [domain=-0.12*pi:-0.06*pi, scale=0.5] plot (\x, {1.3*sin(\x r)+.2});
\draw [domain=-0.02*pi:0.5*pi, scale=0.5] plot (\x, {1.3*sin(\x r)+.2});
\draw [domain=0.5*pi:1.5*pi, scale=0.5] plot (\x, {1.5*sin(\x r)});
\draw [domain=-0.5*pi:1.5*pi, scale=0.5] plot (\x, {-0.4*sin(\x r)-1.9});
\begin{scope}[shift={(-.67,-.59)},scale=.12]
\draw [domain=-0.5*pi:0.5*pi, scale=0.5] plot (\x, {-0.8*sin(\x r)-.1});
\draw [domain=-0.5*pi:0.5*pi, scale=0.5] plot (\x, {0.8*sin(\x r)-1.8});
\end{scope}
\draw [shift={(.14,.04)}, fill=white,white] (1.75,-0.85) rectangle (1.8,-0.9);
\draw [shift={(-.12,-.06)}, fill=white,white] (1.75,-0.85) rectangle (1.8,-0.9);
\draw [fill=white,white] (1.75,-0.85) rectangle (1.8,-0.9);
\end{scope}

\begin{scope}[shift={(.18,-.9)},scale=0.85]
\draw [domain=-0.5*pi:-0.28*pi, scale=0.5] plot (\x, {1.5*sin(\x r)});
\draw [domain=-0.22*pi:-0.12*pi, scale=0.5] plot (\x, {1.5*sin(\x r)});
\draw [domain=-0.08*pi:-0.02*pi, scale=0.5] plot (\x, {1.5*sin(\x r)});
\draw [domain=0.02*pi:0.08*pi, scale=0.5] plot (\x, {1.5*sin(\x r)});
\draw [domain=0.12*pi:1.5*pi, scale=0.5] plot (\x, {1.5*sin(\x r)});
\draw [domain=-0.5*pi:1.5*pi, scale=0.5] plot (\x, {-0.4*sin(\x r)-1.9}) node[right,black,scale=0.9]{$(+1)$};
\draw [shift={(.22,.06)}, fill=white,white] (1.75,-0.85) rectangle (1.8,-0.9);
\draw [shift={(.06,.01)},fill=white,white] (1.75,-0.85) rectangle (1.8,-0.9);
\end{scope}
\begin{scope}[shift={(.18,-1.1)},scale=0.85]
\draw [domain=-0.5*pi:-0.3*pi, scale=0.5] plot (\x, {1.5*sin(\x r)});
\draw [domain=-0.25*pi:-0.18*pi, scale=0.5] plot (\x, {1.5*sin(\x r)});
\draw [domain=-0.135*pi:-0.045*pi, scale=0.5] plot (\x, {1.5*sin(\x r)});
\draw [domain=-0*pi:0.06*pi, scale=0.5] plot (\x, {1.5*sin(\x r)});
\draw [domain=0.1*pi:0.17*pi, scale=0.5] plot (\x, {1.5*sin(\x r)});
\draw [domain=0.21*pi:1.5*pi, scale=0.5] plot (\x, {1.5*sin(\x r)});
\draw [domain=-0.5*pi:1.5*pi, scale=0.5] plot (\x, {-0.4*sin(\x r)-1.9}) node[right,black,scale=0.9]{$(+1)$};
\draw [shift={(.22,.06)}, fill=white,white] (1.75,-0.85) rectangle (1.8,-0.9);
\end{scope}
\begin{scope}[shift={(.18,-1.3)},scale=0.85]
\draw [domain=-0.5*pi:-0.3*pi, scale=0.5] plot (\x, {1.5*sin(\x r)});
\draw [domain=-0.25*pi:-0.2*pi, scale=0.5] plot (\x, {1.5*sin(\x r)});
\draw [domain=-0.135*pi:-0.09*pi, scale=0.5] plot (\x, {1.5*sin(\x r)});
\draw [domain=-0.05*pi:0.04*pi, scale=0.5] plot (\x, {1.5*sin(\x r)});
\draw [domain=0.07*pi:0.15*pi, scale=0.5] plot (\x, {1.5*sin(\x r)});
\draw [domain=0.18*pi:0.27*pi, scale=0.5] plot (\x, {1.5*sin(\x r)});
\draw [domain=0.32*pi:1.5*pi, scale=0.5] plot (\x, {1.5*sin(\x r)});
\draw [domain=-0.5*pi:1.5*pi, scale=0.5] plot (\x, {-0.4*sin(\x r)-1.9}) node[right,black,scale=0.9]{$(+1)$};
\end{scope}
\end{scope}
\end{tikzpicture}
       \caption{\smaller[1]{The standard plumbing graph for the torus bundle $M_0=M(-2;\frac{1}{2},\frac{1}{2},\frac{1}{2},\frac{1}{2})$ (left), and the generalised Lisca-Stipsicz surgery presentation of $\eta$ (right).}}
     \label{bundle}\end{figure}
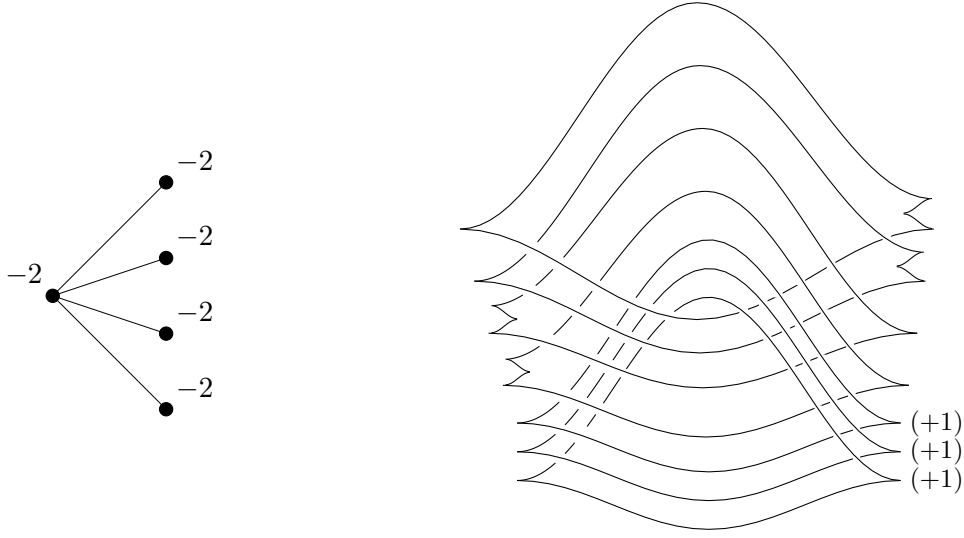

The simplest Stein fillable example is on the Seifert fibred torus bundle $M_0:=M(-2;\frac{1}{2},\frac{1}{2},\frac{1}{2},\frac{1}{2})$, see Figure \ref{bundle} (left); from the classification of Giroux \cite{Giroux2} and Honda \cite{Honda2} we know that this manifold has exactly two Stein fillable contact structures, up to contactomorphism. A negative-twisting structure, which we denote by $\zeta$ and was the one we considered in \cite[Section 6]{CM}, and a structure which is instead zero-twisting and we denote it by $\eta$. The contact 3-manifold $(M_0,\zeta)$ bounds the Stein domain $(W,J)$ given by Legendrianise the plumbing graph of $M_0$, and thus $b_2^0(W)=1$; from \cite[Theorem 4.1]{Etnyre} this means that $\zeta$ is not planar. Conversely, we show that $\eta$ is a generalised Lisca-Stipsicz structure whose surgery presentation is in Figure \ref{bundle} (right), and then it is planar. It is known \cite[Subsection 1.6]{LvHMW} that $(M_0,\eta)$ has the cotangent bundle over the Klein bottle $(D^*K,\dd\lambda_\text{can})$ as a Stein filling; an easy computation gives $\chi(D^*K)=0$ and $b_2(D^*K)=0$ in this case.

\begin{prop}
 The contact structure $\eta$ on $M_0$ is generalised Lisca-Stipsicz. Furthermore, it is supported by a planar open book whose monodromy has positive factorisation.
\end{prop}

From this result it follows that $\eta$ itself contains half convex Giroux torsion; furthermore, any Seifert fibred space $M$ whose standard plumbing graph contains the one of $M_0$ in Figure \ref{bundle} admits a Stein fillable structure with separating half convex Giroux torsion, obtained by Legendrian surgery along some regular or singular fibres of $(M_0,\eta)$.

In order to obtain an actual half Giroux torsion layer it suffices to expand the convex Giroux torsion layer from the proof of \cite[Lemma 2.3]{ACM} by a basic slice, which exists once two same-signed unknots have an additional stabilisation of the same sign. We then need to correct the remaining legs to ensure fillability. The simplest family which can be obtained in this way consists of Seifert fibred spaces whose standard plumbing graph contains the graph of $M(-2;\frac{2}{3},\frac{2}{3},\frac{1}{3},\frac{1}{3})$, see Figure \ref{fig:3} (left); nonetheless, we can give a more general result. 

We denote the Seifert fibred surface bundle $M(-k;s_1,...,s_k,1-s_1,...,1-s_k)$ by $M(s_1,...,s_k)$ for any $s_i\in(0,1)\cap\Q$ with $i=1,...,k$.

\begin{teo}
 \label{teo:convex}
 There are Stein fillable structures with half convex Giroux torsion on any Seifert fibred space $M=M(-k;r_1,...,r_n)$ where $2\leq k\leq\frac{n}{2}$ whose standard graph has the one of $M(s_1,...,s_k)$ as a subgraph. Furthermore, when no $s_i$ is equal to $\frac{1}{2}$ these structures also contain half Giroux torsion.    
\end{teo}

The fact that these manifolds carried universally tight zero-twisting structures without full Giroux torsion was already observed by Massot in \cite[Theorem 9]{Massot}; in fact, it can be shown that the concrete structures we choose in the proof are of this type. However, our construction allows other choices of signs and gives also virtually overtwisted examples.

The structures in Theorems \ref{teo:main} and \ref{teo:convex} are supported by planar open books, whose pages are obtained by stabilising the one of the corresponding $M(s_1,...,s_k)$. Additionally, on the latter manifold the contact structures are \emph{balanced}, that is (up to conjugation) the unknots in the first $k$ legs are fully negatively stabilised and the ones in the last $k$ legs are fully positively stabilised. We show that the corresponding monodromies all have a positive factorisation, providing a Stein filling $(X_M,I_{M,\xi})$ for the given contact structure $(M,\xi)$ when $M$ is a rational homology sphere. We can express the topology of $X_M$ as: \begin{equation}\chi(X_M)=1+b_2(X_M)\hspace{0.5cm}\text{ and }\hspace{0.5cm}b_2(X_M)=|G_M|-|G_{M(s_1,...,s_k)}|-1\label{eq:equation}\end{equation} where $G_M$ and $G_{M(s_1,...,s_k)}$ are the standard plumbing graphs of $M$ and $M(s_1,...,s_k)$ respectively. Note that $X_M$ is necessarily negative-definite because of planarity, see \cite{Etnyre}.

\begin{remark}
 Equation \eqref{eq:equation} implies that we can find Stein domains with any possible second Betti number which have some $(M,\xi)$ with half Giroux torsion as boundary.   
\end{remark} 

Note that if a manifold $M$ in Theorem \ref{teo:convex} were a Brieskorn sphere then by \cite[Lemma 8.1]{CM} its standard graph would be obtained by adding a single vertex to the one of $M(s_1,...,s_k)$, and it would induce the non-canonical orientation. By Equation \eqref{eq:equation} and standard homological arguments $M$ would then bound an integral homology ball, and combining this result with \cite[Corollary 3]{ImC} would give $-k=1-n$. This is not possible because $n\geq2k$; hence, no such $M$ can be a Brieskorn sphere. 

In \cite{ACM} we show that many Brieskorn spheres do not admit rational homology ball symplectic fillings, giving evidence for the validity of a conjecture of Gompf. Our results here show that a counterexample to this conjecture, for the non-canonical orientation, could still arise from a contact 3-manifold with a half Giroux torsion layer, as permitted by \cite[Theorem 1.5]{ACM}.

\subsection*{AI statement} {\smaller[1] The authors used AI only for proof-reading and language purposes.}

\section{Proofs}
\label{section:two}

We first explain how to go from a generalised Lisca-Stipsicz presentation to a compatible planar open book; the procedure is analogous to the one of \cite{Sch} and \cite[Theorem 1.5]{LS.III}, see also \cite[Section 2.1]{M.f}. One $(+1)$-surgery along an unknot with $\tb=-1$ is presented by an annulus with identity monodromy, the other $(+1)$-surgeries manifest in parallel negative Dehn twists along its core. Every other unknot in Figure \ref{fig:SFS} contributes a positive Dehn twist on a stabilised annulus. Concretely, we insert a hole, encircled by one boundary-parallel positive Dehn twist, for every stabilisation of every unknot in the surgery diagram; the stabilisation holes which correspond to positive stabilisations lie between the inner boundary of the annulus and its core, the negative ones between the core and the outer boundary.

Positive Dehn twists corresponding to legs unknots are realised subsequently from centre out.
From the leading unknot of each leg we get a positive Dehn twist along a push-off of the core modified by encircling an additional stabilisation hole for each negative stabilisation, and avoiding a stabilisation hole for each positive stabilisation. 

The twists corresponding to the subsequent unknots in each leg are obtained from a push-off of the twist corresponding to the preceding unknot, additionally encircling all of its negative and avoiding all of its positive stabilisation holes. Hence, a positive Dehn twist corresponding to the unknot $L_j^i$ encircles exactly the negative stabilisation holes of $L_h^i$ for $h\leq j$, together with all positive stabilisation holes except the ones arising from $L_h^i$ for $h\leq j$.

We can now show that all generalised Lisca-Stipsicz tight structures are zero-twisting.

\begin{proof}[Proof of Proposition \ref{prop:zero}]
 To prove that every contact structure of the form in Figure \ref{fig:SFS}(right) is zero-twisting (or overtwisted) we need a Legendrian realisation of the regular fibre whose contact framing agrees with its fibration framing. We recall that a regular fibre of a Seifert fibred space can be seen in the standard graph as a leaf on the central vertex, which in Figure \ref{fig:SFS}(right) corresponds to an additional parallel unknot in the surgery link $T_{1-e_0+n,-(1-e_0+n)}$. In the planar open book as described above this unknot is parallel to the core of the annulus; moreover, two such unknots, corresponding to a Legendrian regular fibre and its fibration push-off, are described by parallel curves on the page, meaning that their page and thus contact framing equals the fibration framing.
\end{proof}

Let us consider any Seifert fibred space $M$ whose standard graph contains the one of $M_0=M(-2;\frac{1}{2},\frac{1}{2},\frac{1}{2},\frac{1}{2})$. We prove that each $M$ carries a generalised Lisca-Stipsicz tight contact structure which is Stein fillable and contains a half convex Giroux torsion layer. 

\begin{prop}
 \label{prop:convex}
 Every $M=M(-2;r_1,...,r_n)$ with $n\geq4$ and $r_i\geq\frac{1}{2}$ for $i=1,...,4$ admits Stein fillable contact structures which contain half convex Giroux torsion. Furthermore, they are supported by a planar open book whose monodromy is obtained from the ones in Figure \ref{fig:factorisation} by positive stabilizations and adding positive Dehn twists. 
\end{prop}

\begin{figure}[ht]
\begin{tikzpicture}[scale=1.3]
\begin{scope}[shift={(-6,0)}]
\begin{scope}[thick,gray!70]
\draw (0,0) circle (2.6cm);\draw (0,0) circle (0.2cm);\draw (0,0.7) circle (0.2cm);\draw (0,-0.7) circle (0.2cm);
\draw (2.1,0) circle (0.2cm);\draw (-2.1,0) circle (0.2cm);
\end{scope}

\draw[thick,dashed,color={r}] (0,0) circle (1.5cm);
\draw[thick,color={r}] (0,-0.7) circle (0.25cm);
\draw[thick,color={r}] (2.1,0) circle (0.25cm);
\draw[thick,color={r}] (0,0.35) circle (0.7cm);
\draw[thick,color={r}] (.4,0) circle (2.1cm);

\draw[thick,dashed,color={b}] (0,0) circle (1.3cm);
\draw[thick,color={b}] (0,0.7) circle (0.25cm);
\draw[thick,color={b}] (-2.1,0) circle (0.25cm);
\draw[thick,color={b}] (0,-0.35) circle (0.7cm);
\draw[thick,color={b}] (-.4,0) circle (2.1cm);

\draw[white] (0,-1.5) circle (0pt) node[below,black,scale=0.7]{$N$};
\draw[white] (0.65,-0.45) circle (0pt) node[right,black,scale=0.7]{$Q$};
\draw[white] (0.65,0.45) circle (0pt) node[right,black,scale=0.7]{$P$};
\draw[white] (0.22,-0.7) circle (0pt) node[right,black,scale=0.7]{$p$};
\draw[white] (0.22,0.7) circle (0pt) node[right,black,scale=0.7]{$q$};
\draw[white] (0.6,-2.1) circle (0pt) node[below,black,scale=0.7]{$R$};
\draw[white] (-0.6,-2.1) circle (0pt) node[below,black,scale=0.7]{$S$};
\draw[white] (2.1,-.25) circle (0pt) node[below,black,scale=0.7]{$r$};
\draw[white] (-2.1,-.25) circle (0pt) node[below,black,scale=0.7]{$s$};
\end{scope}

\begin{scope}[thick,gray!70]
\draw (0,0) circle (2.6cm);\draw (0,0) circle (0.2cm);\draw (.1,0.8) circle (0.2cm);\draw (-.1,-0.8) circle (0.2cm);
\draw (2.1,0) circle (0.2cm);\draw (-2.1,0) circle (0.2cm);
\end{scope}

\draw[thick,color={b},rotate around={165:(-.17,.3)}] (.5,1) ellipse (1.5cm and 1cm);
\draw[thick,color={r},rotate around={165:(-.17,.3)}] (-1.3,0.1) ellipse (1.5cm and 1cm);
\draw[thick,color={b},rotate around={20:(.7,.7)}] (-1,1) ellipse (1.45cm and .5cm);
\draw[thick,color={r},rotate around={20:(.7,.7)}] (.6,-.45) ellipse (1.45cm and .5cm);

\draw[white] (-0.8,-1.4) circle (0pt) node[right,black,scale=0.7]{$B$};
\draw[white] (-0.8,1.2) circle (0pt) node[right,black,scale=0.7]{$b$};
\draw[white] (0.8,1.4) circle (0pt) node[right,black,scale=0.7]{$A$};
\draw[white] (0.8,-1.2) circle (0pt) node[right,black,scale=0.7]{$a$};
\end{tikzpicture}
\caption{\smaller[1]{Planar monodromy for $\eta$ on $M(-2;\frac{1}{2},\frac{1}{2},\frac{1}{2},\frac{1}{2})$ and its positive factorisation. In gray are boundary components of the punctured disk. The full curves correspond to positive Dehn twists, the dashed curves represent two copies of the negative Dehn twist along $N$. Colours denote curves involved in a single application of the lantern relation.}}
\label{fig:factorisation}
\end{figure}
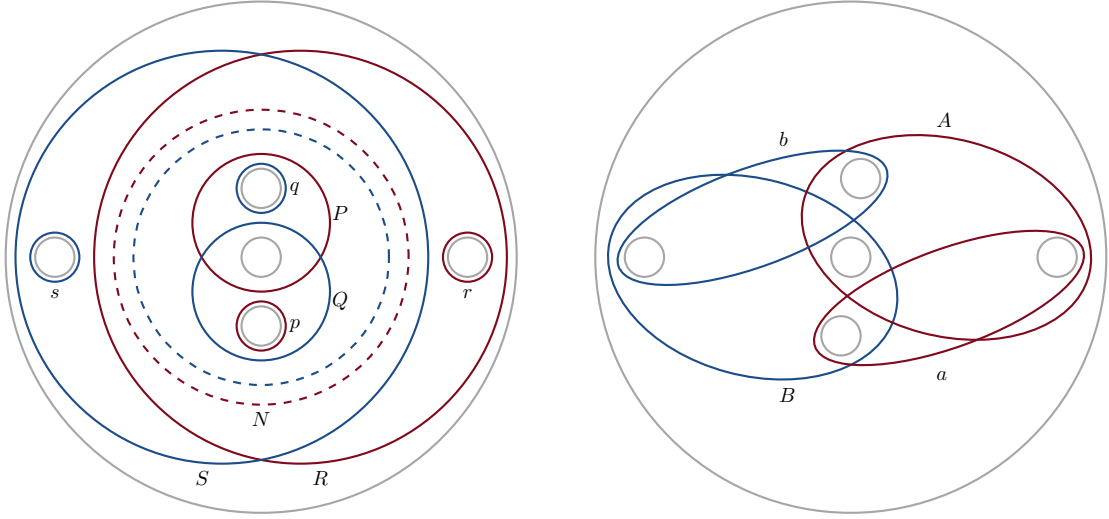

\begin{proof}
 We start from the structure on $M_0$ in Figure \ref{bundle}; since it is generalised Lisca-Stipsicz, it is zero-twisting and contains half convex Giroux torsion by Corollary \ref{cor:zero}. To prove it is also Stein fillable, we consider the corresponding open book on the left in Figure \ref{fig:factorisation}. By applying the lantern relation twice as in the picture, we obtain that the monodromy of the open book has a positive factorisation; indeed, denoting positive Dehn twists with the same letters as the curves, we have \[N^{-2}pPrRqQsS=\left(N^{-1}pPrR\right)\left(N^{-1}qQsS\right)=aAbB\] because $N$ commutes with all the twists in the left picture; therefore, the structure is Stein fillable. Note that, as mentioned in the introduction, the manifold $M_0$ carries a unique Stein fillable planar structure $\eta$, up to contactomorphism; hence, we just proved that the surgery presentation on the right in Figure \ref{bundle} represents $\eta$.

 Since each $M$ has a generalised Lisca-Stipsicz presentation obtained by performing $(-1)$-contact surgeries along some Legendrian knots in $(M_0,\eta)$, the resulting structure remains Stein fillable with half convex Giroux torsion. Indeed, the corresponding positive Dehn twists can be placed on the stabilised planar page after the lantern substitution of Figure \ref{fig:factorisation} has been applied.
\end{proof}

\begin{figure}[ht]
\begin{tikzpicture}[scale=1] 
      \tkzDefPoints{0/0/A, 1.5/1.5/B, 1.5/-1.5/D, 1.5/0.5/C, 1.5/-0.5/G, 3/1.5/E, 3/0.5/F, 0/-3/O}      
       \tkzDrawSegment(A,B)\tkzDrawSegment(A,D)\tkzDrawSegment(A,C)\tkzDrawSegment(A,G)\tkzDrawSegment(B,E)\tkzDrawSegment(C,F)
    \tkzDrawPoints[fill,black,size=5](A,B,C,D,G,E,F)\tkzDrawPoints[fill,white,size=5](O)
     \tkzLabelPoint[above left](A){$-2$} \tkzLabelPoint[above right](B){$-2$}\tkzLabelPoint[above right](C){$-2$}
     \tkzLabelPoint[above right](D){$-3$}\tkzLabelPoint[above right](G){$-3$}\tkzLabelPoint[above right](E){$-2$}\tkzLabelPoint[above right](F){$-2$}
\end{tikzpicture}\hspace{3cm}
\begin{tikzpicture}
\begin{scope}[scale=1.9]
\begin{scope}[shift={(-.06,.6)}, scale=1.05]
\draw [domain=-0.5*pi:0.5*pi, scale=0.5] plot (\x, {1.5*sin(\x r)});
\draw [domain=0.5*pi:1.48*pi, scale=0.5] plot (\x, {1.1*sin(\x r)+0.4});
\draw [domain=-0.5*pi:1.5*pi, scale=0.5] plot (\x, {-0.6*sin(\x r)-2.1});
\begin{scope}[shift={(2.25,-.39)},scale=.12]
\draw [domain=-0.5*pi:0.5*pi, scale=0.5] plot (\x, {0.8*sin(\x r)-.1});
\draw [domain=-0.5*pi:0.5*pi, scale=0.5] plot (\x, {-0.8*sin(\x r)-1.8});
\end{scope}
\begin{scope}[shift={(2.25,-.59)},scale=.12]
\draw [domain=-0.5*pi:0.5*pi, scale=0.5] plot (\x, {0.8*sin(\x r)-.1});
\draw [domain=-0.5*pi:0.5*pi, scale=0.5] plot (\x, {-0.8*sin(\x r)-1.8});
\end{scope}
\draw [shift={(.3,.07)}, fill=white,white] (1.75,-0.85) rectangle (1.8,-0.9);
\draw [shift={(-.1,-.12)}, fill=white,white] (1.75,-0.85) rectangle (1.8,-0.9);
\draw [shift={(-.3,-.22)}, fill=white,white] (1.75,-0.85) rectangle (1.8,-0.9);
\draw [shift={(-.45,-.3)}, fill=white,white] (1.75,-0.85) rectangle (1.8,-0.9);
\draw [shift={(-.55,-.37)}, fill=white,white] (1.75,-0.85) rectangle (1.8,-0.9);
\draw [shift={(-.68,-.43)}, fill=white,white] (1.75,-0.85) rectangle (1.8,-0.9);
\end{scope}

\begin{scope}[shift={(0,.2)}]
\draw [domain=-0.5*pi:-0.27*pi, scale=0.5] plot (\x, {1.5*sin(\x r)});
\draw [domain=-0.22*pi:0.5*pi, scale=0.5] plot (\x, {1.5*sin(\x r)});
\draw [domain=0.5*pi:1.48*pi, scale=0.5] plot (\x, {1.25*sin(\x r)+0.25});
\draw [domain=-0.5*pi:0.5*pi, scale=0.5] plot (\x, {-0.5*sin(\x r)-2});
\draw [domain=0.5*pi:1.5*pi, scale=0.5] plot (\x, {-0.35*sin(\x r)-2.15});
\begin{scope}[shift={(2.25,-.74)},scale=.12]
\draw [domain=-0.5*pi:0.5*pi, scale=0.5] plot (\x, {0.8*sin(\x r)-.1});
\draw [domain=-0.5*pi:0.5*pi, scale=0.5] plot (\x, {-0.8*sin(\x r)-1.8});
\end{scope}
\begin{scope}[shift={(2.25,-.54)},scale=.12]
\draw [domain=-0.5*pi:0.5*pi, scale=0.5] plot (\x, {0.8*sin(\x r)-.1});
\draw [domain=-0.5*pi:0.5*pi, scale=0.5] plot (\x, {-0.8*sin(\x r)-1.8});
\end{scope}
\draw [shift={(.19,-.07)}, fill=white,white] (1.75,-0.85) rectangle (1.8,-0.9);
\draw [shift={(-.09,-.15)}, fill=white,white] (1.75,-0.85) rectangle (1.8,-0.9);
\draw [shift={(-.275,-.215)}, fill=white,white] (1.75,-0.85) rectangle (1.8,-0.9);
\draw [shift={(-.39,-.245)}, fill=white,white] (1.75,-0.85) rectangle (1.8,-0.9);
\draw [shift={(-.5,-.28)}, fill=white,white] (1.75,-0.85) rectangle (1.8,-0.9);
\end{scope}

\begin{scope}[shift={(.06,-.2)},scale=0.95]
\draw [domain=-0.48*pi:-0.34*pi, scale=0.5] plot (\x, {1.3*sin(\x r)+.2});
\draw [domain=-0.28*pi:-0.2*pi, scale=0.5] plot (\x, {1.3*sin(\x r)+.2});
\draw [domain=-0.15*pi:.5*pi, scale=0.5] plot (\x, {1.3*sin(\x r)+.2});
\draw [domain=0.5*pi:1.5*pi, scale=0.5] plot (\x, {1.5*sin(\x r)});
\draw [domain=-0.5*pi:1.5*pi, scale=0.5] plot (\x, {-0.4*sin(\x r)-1.9});
\begin{scope}[shift={(-.67,-.59)},scale=.12]
\draw [domain=-0.5*pi:0.5*pi, scale=0.5] plot (\x, {-0.8*sin(\x r)-.1});
\draw [domain=-0.5*pi:0.5*pi, scale=0.5] plot (\x, {0.8*sin(\x r)-1.8});
\end{scope}
\draw [shift={(.09,.02)}, fill=white,white] (1.75,-0.85) rectangle (1.8,-0.9);
\draw [shift={(-.12,-.06)}, fill=white,white] (1.75,-0.85) rectangle (1.8,-0.9);
\draw [shift={(-.23,-.09)}, fill=white,white] (1.75,-0.85) rectangle (1.8,-0.9);
\draw [shift={(-.34,-.12)}, fill=white,white] (1.75,-0.85) rectangle (1.8,-0.9);
\draw [shift={(.4,0.1)},fill=white,white] (1.75,-0.85) rectangle (1.8,-0.9);
\begin{scope}[shift={(2.25,-.65)},scale=.18]
\draw [domain=-0.5*pi:-0.3*pi, scale=0.5] plot (\x, {0.8*sin(\x r)-.1});
\draw [domain=-0.1*pi:1.5*pi, scale=0.5] plot (\x, {0.8*sin(\x r)-.1});
\draw [domain=-0.5*pi:1.5*pi, scale=0.5] plot (\x, {-0.8*sin(\x r)-1.8});
\end{scope}
\end{scope}

\begin{scope}[shift={(.12,-.6)},scale=0.9]
\draw [domain=-0.48*pi:-0.32*pi, scale=0.5] plot (\x, {1.3*sin(\x r)+.2});
\draw [domain=-0.25*pi:-0.16*pi, scale=0.5] plot (\x, {1.3*sin(\x r)+.2});
\draw [domain=-0.12*pi:-0.06*pi, scale=0.5] plot (\x, {1.3*sin(\x r)+.2});
\draw [domain=-0.02*pi:0.5*pi, scale=0.5] plot (\x, {1.3*sin(\x r)+.2});
\draw [domain=0.5*pi:1.5*pi, scale=0.5] plot (\x, {1.5*sin(\x r)});
\draw [domain=-0.5*pi:1.5*pi, scale=0.5] plot (\x, {-0.4*sin(\x r)-1.9});
\begin{scope}[shift={(-.67,-.59)},scale=.12]
\draw [domain=-0.5*pi:0.5*pi, scale=0.5] plot (\x, {-0.8*sin(\x r)-.1});
\draw [domain=-0.5*pi:0.5*pi, scale=0.5] plot (\x, {0.8*sin(\x r)-1.8});
\end{scope}
\draw [shift={(.14,.04)}, fill=white,white] (1.75,-0.85) rectangle (1.8,-0.9);
\draw [shift={(-.12,-.06)}, fill=white,white] (1.75,-0.85) rectangle (1.8,-0.9);
\draw [fill=white,white] (1.75,-0.85) rectangle (1.8,-0.9);
\draw [shift={(.4,0.1)},fill=white,white] (1.75,-0.85) rectangle (1.8,-0.9);
\begin{scope}[shift={(2.25,-.65)},scale=.18]
\draw [domain=-0.5*pi:-0.3*pi, scale=0.5] plot (\x, {0.8*sin(\x r)-.1});
\draw [domain=-0.1*pi:1.5*pi, scale=0.5] plot (\x, {0.8*sin(\x r)-.1});
\draw [domain=-0.5*pi:1.5*pi, scale=0.5] plot (\x, {-0.8*sin(\x r)-1.8});
\end{scope}
\end{scope}

\begin{scope}[shift={(.18,-.9)},scale=0.85]
\draw [domain=-0.5*pi:-0.28*pi, scale=0.5] plot (\x, {1.5*sin(\x r)});
\draw [domain=-0.22*pi:-0.12*pi, scale=0.5] plot (\x, {1.5*sin(\x r)});
\draw [domain=-0.08*pi:-0.02*pi, scale=0.5] plot (\x, {1.5*sin(\x r)});
\draw [domain=0.02*pi:0.08*pi, scale=0.5] plot (\x, {1.5*sin(\x r)});
\draw [domain=0.12*pi:1.5*pi, scale=0.5] plot (\x, {1.5*sin(\x r)});
\draw [domain=-0.5*pi:1.5*pi, scale=0.5] plot (\x, {-0.4*sin(\x r)-1.9}) node[right,black,scale=0.9]{$(+1)$};
\draw [shift={(.22,.06)}, fill=white,white] (1.75,-0.85) rectangle (1.8,-0.9);
\draw [shift={(.06,.01)},fill=white,white] (1.75,-0.85) rectangle (1.8,-0.9);
\end{scope}
\begin{scope}[shift={(.18,-1.1)},scale=0.85]
\draw [domain=-0.5*pi:-0.3*pi, scale=0.5] plot (\x, {1.5*sin(\x r)});
\draw [domain=-0.25*pi:-0.18*pi, scale=0.5] plot (\x, {1.5*sin(\x r)});
\draw [domain=-0.135*pi:-0.045*pi, scale=0.5] plot (\x, {1.5*sin(\x r)});
\draw [domain=-0*pi:0.06*pi, scale=0.5] plot (\x, {1.5*sin(\x r)});
\draw [domain=0.1*pi:0.17*pi, scale=0.5] plot (\x, {1.5*sin(\x r)});
\draw [domain=0.21*pi:1.5*pi, scale=0.5] plot (\x, {1.5*sin(\x r)});
\draw [domain=-0.5*pi:1.5*pi, scale=0.5] plot (\x, {-0.4*sin(\x r)-1.9}) node[right,black,scale=0.9]{$(+1)$};
\draw [shift={(.22,.06)}, fill=white,white] (1.75,-0.85) rectangle (1.8,-0.9);
\end{scope}
\begin{scope}[shift={(.18,-1.3)},scale=0.85]
\draw [domain=-0.5*pi:-0.3*pi, scale=0.5] plot (\x, {1.5*sin(\x r)});
\draw [domain=-0.25*pi:-0.2*pi, scale=0.5] plot (\x, {1.5*sin(\x r)});
\draw [domain=-0.135*pi:-0.09*pi, scale=0.5] plot (\x, {1.5*sin(\x r)});
\draw [domain=-0.05*pi:0.04*pi, scale=0.5] plot (\x, {1.5*sin(\x r)});
\draw [domain=0.07*pi:0.15*pi, scale=0.5] plot (\x, {1.5*sin(\x r)});
\draw [domain=0.18*pi:0.27*pi, scale=0.5] plot (\x, {1.5*sin(\x r)});
\draw [domain=0.32*pi:1.5*pi, scale=0.5] plot (\x, {1.5*sin(\x r)});
\draw [domain=-0.5*pi:1.5*pi, scale=0.5] plot (\x, {-0.4*sin(\x r)-1.9}) node[right,black,scale=0.9]{$(+1)$};
\end{scope}
\end{scope}
\end{tikzpicture}
\caption{\smaller[1]{The contact structure $\eta'$ on $M'=M(-2;\frac{2}{3},\frac{2}{3},\frac{1}{3},\frac{1}{3})$.}}
\label{fig:3}
\end{figure}
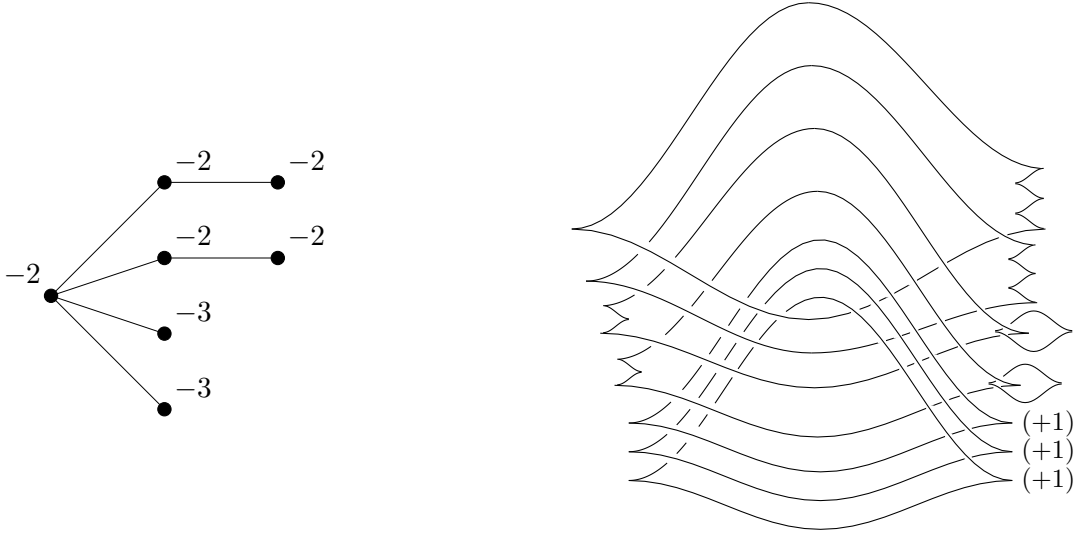

We now modify the above example to get actual half Giroux torsion.

\begin{prop}
 \label{prop:actual}
 Every $M=M(-2;r_1,...,r_n)$ with $n\geq4$ whose standard graph contains the subgraph of $M':=M(-2;\frac{2}{3},\frac{2}{3},\frac{1}{3},\frac{1}{3})$ in Figure \ref{fig:3} admits Stein fillable contact structures with half Giroux torsion. 
\end{prop}

This result implies our main theorem.

\begin{proof}[Proofs of Proposition \ref{prop:actual} and Theorem \ref{teo:main}]
    We recall from the proof of \cite[Lemma 2.3]{ACM} how the vertical half convex Giroux torsion layer embeds in a convex decomposition of a zero-twisting structure with $e_0<-1$. Namely, we obtain the two neighbouring thickened tori with slopes $-1$ and $\infty$, attached along the $\infty$-slope, by connecting all the same-signed outermost basic slices of singular fibres. 

    There is actual half Giroux torsion when the twisting extends over the convex boundaries; we prove this by noticing that there is an additional basic slice attached to the half convex Giroux torsion layer, as in \cite[Proposition 2.11]{MV}. For this basic slice to exist we need that all singular fibres with outermost slices of one sign have additional same sign stabilisations on the leading unknot. Concretely, on $M'$ as in Figure \ref{fig:3} we have $\rot(L_1^3)=\rot(L_1^4)=-2$, and we can cut between the pairs of outermost slices to obtain a torus of slope $-\frac{1}{2}$ around the two singular fibres. Seen from the torus supporting the torsion, this slope corresponds to slope $\frac{1}{2}-2=-\frac{3}{2}$ because of the opposite orientation and $e_0=-2$ normalisation. The two neighbouring thickened tori thus have boundary slopes $-1,\infty$ and $\infty,-\frac{3}{2}$. They glue together into a rotative layer with boundary slopes $-1$ and $-\frac{3}{2}$ which contains $\infty$-slope, and hence contains half Giroux torsion.

    However, to keep fillability we now need additional surgeries on the other two singular fibres. The proof that the contact structure in Figure \ref{fig:3} is indeed Stein fillable comes by the positive factorisation of the monodromy, analogous to the case of $(M_0,\eta)$ except that the two lantern substitutions are replaced by two daisy substitutions of \cite[Lemma 3.5]{PvHM}. This proves that the contact structure $\eta'$ on $M'$ is Stein fillable with half Giroux torsion.

    As before, Stein fillable structures on any $M$ are obtained by Legendrian surgery on $(M',\eta')$. To preserve half Giroux torsion we restrict to the ones for which all the fully negative leading unknots have at least two stabilisations. This choice is of course possible for every $M$; in particular, we can take all additional surgeries fully positive. In this case, the half Giroux torsion layer is separating the two fibres with fully negative $(-3)$-leading unknots on one and all the other fibres on the other side. We conclude that every $M$ admits Stein fillable structure with half Giroux torsion.
\end{proof}

The construction described in the above proof can be replicated also for any Seifert fibred space $M$ whose standard plumbing graph contains as a subgraph the standard graph of $M(s_1,...,s_k)$, which denotes the manifold $M(-k;s_1,...,s_k,1-s_1,...,1-s_k)$ for $k\geq2$. 

\begin{proof}[Proof of Theorem \ref{teo:convex}]
 We write out the proof for $M(s_1,\dots,s_k)$; the result for any $M$ then follows the same as in Proposition \ref{prop:actual}. If we choose all stabilisations on the first $k$ legs of the same sign and all stabilisations on the other legs of the opposite sign (in other words, if the contact structure on this surface bundle is balanced), the structure is Stein fillable. Indeed, we get positive factorisation by $k$-times applying the substitution of \cite[Proposition 3.3]{M.f} because all twists on positive legs commute with all twists on negative legs and $N$ commutes with everything.
 
 The first part of the statement now follows directly from Stein fillability and Corollary \ref{cor:zero}. Furthermore, when $s_i>\frac{1}{2}$, we have an actual half Giroux torsion layer because $1-s_i<\frac{1}{2}$ and therefore their leading unknots have $m_1^i<-2$; in particular, they have at least two negative stabilisations. The layer with boundary slopes $-1$ and $\frac{1}{2}-k$, which contains half Giroux torsion, is then constructed as in the proof of Proposition \ref{prop:actual}.
\end{proof}

Consider the Stein filling $(X_M,I_{M,\xi})$ of any $(M,\xi)$ rational homology sphere as in Theorem \ref{teo:convex} given by the positive factorisation we described: the fact that $(X_M,I_{M,\xi})$ satisfies $\chi(X_M)=1+b_2(X_M)$ follows by homological arguments. Moreover, from the factorisation we immediately see that each additional Legendrian surgery increases by one the Euler characteristic of the filling, starting from the filling $D$ of $M(s_1,...,s_k)$ from the positive planar open book. We obtain Equation \eqref{eq:equation} with the following lemma.

\begin{lemma}
 For the $4$-manifold $D$ as above one has that \[\chi(D)=1-|G_{M(s_1,...,s_k)}|+(|G_{M(s_1,...,s_k)}|-1)=0\] where $G_{M(s_1,...,s_k)}$ denotes the standard graph of $M(s_1,...,s_k)$.   
\end{lemma}
\begin{proof}
 We prove the result for $k=1$, for the other cases the same computation is repeated $k$-times. We rely on the proof of \cite[Proposition 3.3]{M.f} and we use the same notations. We can write \[-\frac{1}{s}=[-2^{\times b_1},-b_2-3,-2^{\times b_3},\dots,-b_m-2]\:\text{ and }\:-\frac{1}{1-s}=[-b_1-3,-2^{\times b_2},-b_3-3,\dots,-2^{\times b_m}]\] for some non-negative integers $b_1,...,b_m$. Then the page has $1+\sum_{j=1}^m(b_j+1)$ holes, that is one more than the number of stabilisations for the two legs with framings $s$ and $1-s$, one negative Dehn twist and $2\sum_{j=1}^m(b_j+1)$ positive Dehn twists; among the latter ones, there are $\sum_{j=1}^m(b_j+1)$ boundary parallel twists, while $\sum_{j=1}^m(b_j+1)$ come from the unknots in the two chains. 
 
 To obtain a positive factorisation, we need to apply $m$ daisy relations \cite[Lemma 3.5]{PvHM} and the $j$-th daisy has $b_j+2$ petals (except the first which has $b_1+1$); here, the lantern relation is just a daisy relation with two petals. One application of the daisy relation with $b$ petals reduces the difference between the number of positive and negative Dehn twists by $b-1$; therefore, in the final positive factorisation the number of positive twists is reduced by $1+b_1+\sum_{j=2}^{m}(b_j+1)$. Thus, we end up with $1+\sum_{j=1}^m(b_j+1)$ holes, that is the number of $1$-handles, and $\sum_{j=1}^m(b_j+1)$ positive Dehn twists, that is the number of $2$-handles; finally, the number of vertices equal $|G_{M(s_1,...,s_k)}|=1+\sum_{j=1}^m(b_j+1)$, giving the formula for Euler characteristic.
\end{proof}

\end{document}